\documentclass[reqno, 10pt]{amsart}
\usepackage{amssymb}
\usepackage{amsmath}
\usepackage[mathscr]{euscript}
\usepackage{hyperref}
\usepackage{graphicx}
\usepackage[normalem]{ulem}

\usepackage{amssymb}
\usepackage{hyperref}
\RequirePackage[dvipsnames]{xcolor} 
\definecolor{halfgray}{gray}{0.55} 
\definecolor{webgreen}{rgb}{0,0.5,0}
\definecolor{webbrown}{rgb}{.6,0,0} \hypersetup{%
  colorlinks=true, linktocpage=true, pdfstartpage=3,
  pdfstartview=FitV,%
  breaklinks=true, pdfpagemode=UseNone, pageanchor=true,
  pdfpagemode=UseOutlines,%
  plainpages=false, bookmarksnumbered, bookmarksopen=true,
  bookmarksopenlevel=1,%
  hypertexnames=true,
  pdfhighlight=/O,
  urlcolor=webbrown, linkcolor=RoyalBlue,
  citecolor=webgreen, 
  pdftitle={},%
  pdfauthor={},%
  pdfsubject={2000 MAthematical Subject Classification: Primary:},%
  pdfkeywords={},%
  pdfcreator={pdfLaTeX},%
  pdfproducer={LaTeX with hyperref}%
}

\usepackage{enumitem}

\newtheorem{theorem}{Theorem}[section]
\newtheorem{lemma}[theorem]{Lemma}
\newtheorem{corollary}[theorem]{Corollary}
\newtheorem{proposition}[theorem]{Proposition}
\newtheorem{claim}[theorem]{Claim}

\def\G{\mathcal{G}}
\def\R{\mathbb{R}}

\def\Reg{\mathcal{R}^{\mu}}
\def\Pd{\mathbb{P}^{d-1}}

\def\Pr{\mathbb{P}}
\newcommand{\norm}[1]{{\left\lVert \, #1 \, \right\rVert}}
\def\exteriorj{\Lambda^{j}\left(\mathbb{R}^d\right)}
\def\grassj{\textrm{Grass}(j,d)}
\def\d{\operatorname{dist}}

\begin{document}

\title
[Periodic approximation of Oseledets subspaces]
{Periodic approximation of Oseledets subspaces for semi-invertible cocycles}

\author{Lucas Backes}

\address{\noindent Departamento de Matem\'atica, Universidade Federal do Rio Grande do Sul, Av. Bento Gon\c{c}alves 9500, CEP 91509-900, Porto Alegre, RS, Brazil.
\newline e-mail: \rm
  \texttt{lhbackes@impa.br} }

\date{\today}


\begin{abstract}
We prove that, for semi-invertible linear cocycles, Oseledets subspaces associated to ergodic measures may be approximated by Oseledets subspaces associated to periodic points.
\end{abstract}

\maketitle

\section{Introduction}

Since its introduction by Smale in \cite{Sm67}, the notion of hyperbolicity has played a major rule in the study of Dynamical Systems. One of the main features exhibited by such systems is the abundance of periodic points and, as a consequence, many of its interesting dynamical properties can be described in terms of the information given on such periodic points. For instance, it is known that cohomology classes of H\"older cocycles over hyperbolic systems are characterized by its information on periodic points (see for instance \cite{Liv71, Liv72, Kal11, dLW10, Bac15, Sa15, BK16, KP16} and references therein), equilibrium states associated to different potentials coincide whenever those potentials have the same information on periodic points \cite{Bow75} and the information carried by the Lyapunov exponents is concentrated on periodic points \cite{Dai10, WS10, Kal11, Bac}.

In this note we address the problem of extracting information from periodic points in the context of Oseledets splittings of semi-invertible linear cocycles. As a consequence of our main result we get that (see Section \ref{subsec: main} for precise statements)
\begin{theorem}
If $f:M\to M$ is a homeomorphism satisfying the Anosov Closing property and $A:M\to M(d,\mathbb{R})$ is a $\alpha$-H\"older continuous map then the Oseledets subspaces of $(A,f)$ associated to ergodic measures can be approximated by Oseledets subspaces of $(A,f)$ associated to periodic points.
\end{theorem}
This problem was already addressed by \cite{LLS09} in the context of $C^{1+r}$ non-uniformly hyperbolic systems with simple Lyapunov spectrum and was latter generalized by \cite{LLS14} to any $C^{1+r}$ non-uniformly hyperbolic system (that is, with no simplicity assumption). While both works dealt only with the case of derivative cocycles with no zero Lyapunov exponents (which is a particular example of invertible cocycle) we treat the broader case of semi-invertible cocycles.

\section{Statements }

Let $(M,d)$ be a compact metric space, $\mu$ a measure defined on the Borel sets of $(M,d)$ and $f: M \to M $ a measure preserving homeomorphism. Assume also that $\mu$ is ergodic. 
\subsection{Semi-invertible cocycles, Lyapunov exponents and Oseledets splittings}

Given a measurable matrix-valued map $A:M\rightarrow M(d, \mathbb{R})$, the pair $(A,f)$ is called a \textit{semi-invertible linear cocycle} (or just \textit{linear cocycle} for short). Sometimes one calls linear cocycle (over $f$ generated by $A$), instead, the sequence $\lbrace A^n\rbrace _{n\in \mathbb{N}}$ defined by
\begin{equation}\label{def:cocycles}
A^n(x)=
\left\{
	\begin{array}{ll}
		A(f^{n-1}(x))\ldots A(f(x))A(x)  & \mbox{if } n>0 \\
		Id & \mbox{if } n=0 \\
	\end{array}
\right.
\end{equation}
for all $x\in M$. The word `semi-invertible' refers to the fact that the action of the underlying dynamical system $f$ is invertible while the action on the fibers given by $A$ may fail to be invertible. We refer to the Introduction of \cite{DrF} for some interesting applications of semi-invertible cocycles.  

Assuming $\int \log ^{+}\norm{A(x)}d\mu (x)<\infty$, it was proved in \cite{FLQ10} that there exists a full $\mu$-measure set $\Reg \subset M$, whose points are called $\mu$-regular points, such that for every $x\in \Reg$ there exist numbers $\lambda _1>\ldots > \lambda _{l}\geq -\infty$, called \textit{Lyapunov exponents}, and a direct sum decomposition $\mathbb{R}^d=E^{1,A}_{x}\oplus \ldots \oplus E^{l,A}_{x}$ into vector subspaces which are called \textit{Oseledets subspaces} and depend measurable on $x$ such that, for every $1\leq i \leq l$,
\begin{itemize}
\item dim$(E^{i,A}_{x})$ is constant,
\item $A(x)E^{i,A}_{x}\subseteq E^{i,A}_{f(x)}$ with equality when $\lambda _i>-\infty$
\end{itemize}
and
\begin{itemize}
\item $\lambda _i =\lim _{n\rightarrow +\infty} \dfrac{1}{n}\log \parallel A^n(x)v\parallel$ 
for every non-zero $v\in E^{i,A}_{x}$.
\end{itemize}
This result extends a famous theorem due to Oseledets \cite{Ose68} known as the \textit{multiplicative ergodic theorem} which was originally stated in both, \textit{invertible} (both $f$ and the matrices are assumed to be invertible) and \textit{non-invertible} (neither $f$ nor the matrices are assumed to be invertible) settings (see also \cite{LLE}). While in the invertible case the conclusion is similar to the conclusion above (except that all Lyapunov exponents are finite), in the non-invertible case, instead of a direct sum decomposition into invariant vector subspaces, one only get an invariant \textit{filtration} (a sequence of nested subspaces) of $\R ^d$. 

Let us denote by 
$$\gamma _1(A,\mu)\geq \gamma _2(A,\mu)\geq \ldots \geq \gamma _d(A,\mu)$$
the Lyapunov exponents of $(A,f)$ with respect to the measure $\mu$ counted with multiplicities. Given a periodic point $p$, we denote the Lyapunov exponents counted with multiplicities of $(A,f)$ at $p$ by $\{\gamma _i(A,p)\}_{i=1}^d$. When there is no risk of ambiguity, we suppress the index $A$ or even both $A$ and $\mu$ from the previous objects.

\subsection{Angle between subspaces} The \textit{angle} $\measuredangle{(E,F)}$ between two subspaces $E$ and $F$ of $\mathbb{R} ^d$ is defined as follows: given $w\in \mathbb{R}^d$ we define 
$$\d(w,E)=\inf_{v\in E} \norm{ w-v}.$$ 
It is easy to see that $\d(w,E)=\norm{w^{\perp}}$ where $w^{\perp}=w-\mbox{Proj}_{E}w$. More generally, we may consider the distance between $E$ and $F$ given by
\begin{equation}\label{def:distancia}
\d(E,F)=\sup_{v\in E,w\in F}\left\{ \d\left(\frac{v}{\norm{v}},F\right),\d\left(\frac{w}{\norm{w}},E\right)\right\}.
\end{equation}
Then, the angle between $E$ and $F$ is just $\measuredangle(E,F)=\sin^{-1}(\d(E,F))$. 

\subsection{Periodic approximation properties}
We say that $(A,f,\mu)$ has the \textit{periodic approximation property for the Lyapunov exponents} if there exists a sequence $(p_k)_{k\in \mathbb{N}}$ of periodic points satisfying
\begin{equation}\label{eq: mu_pk to mu}
\mu _{p_k} :=\dfrac{1}{n_k}\sum _{j=0}^{n_k-1}\delta _{f^j(p_k)}\xrightarrow[k\to \infty]{\text{weak}^*}\mu
\end{equation}
where $n_k$ is the $f$-period of $p_k$ and such that
\begin{equation}\label{eq: gamma pk to gamma mu}
\gamma_j (A,p_k)\xrightarrow{k\to \infty}\gamma_j (A,\mu)
\end{equation}
for every $j=1,\ldots,d$. 

Similarly, $(A,f,\mu)$ is said to have the \textit{periodic approximation property for the Oseledets splitting} if there exists a sequence $(p_k)_{k\in \mathbb{N}}$ of periodic points satisfying \eqref{eq: mu_pk to mu} and, moreover, for each $k\in \mathbb{N}$ there exists a set $\G_k\subset M$ with $\mu (\G_k)>1-\frac{1}{k}$ so that for every $x\in \G_k$ there exists $q\in \text{orb}(p_k)$ satisfying
$$\measuredangle(E^{j,A}_x,F^{j,A}_{q})<\frac{1}{k}$$
for every $j=1,\ldots ,l$ where $F^{j,A}_{q}$ is the sum of Oseledets subspaces of $A$ at $q$ associated to consecutive Lyapunov exponents. More precisely, if $\gamma_{i_1}(A,\mu)\geq\cdots \geq \gamma _{i_t}(A,\mu)$ are the Lyapunov exponents associated to $E^{j,A}_x$ then $F^{j,A}_{q}$ is the sum of the Oseledets subspaces associated to $\gamma_{i_1}(A,p_k)\geq\cdots \geq \gamma _{i_t}(A,p_k)$.

\subsection{Main result} \label{subsec: main}
The main result of this note is the following one:

\begin{theorem} \label{theo: main} 
Let $f: M\to M $ be a homeomorphism, $\mu$ an ergodic $f$-invariant probability measure and $A:M\to M(d,\R)$ a continuous map. If the system $(A,f,\mu)$ has the periodic approximation property for the Lyapunov exponents then it also has the periodic approximation property for the Oseledets splitting. 
\end{theorem}

In view of the previous result, it is natural then to ask whether the converse statement is also true. More precisely, if the system $(A,f,\mu)$ has the periodic approximation property for the Oseledets splitting then does it also have the periodic approximation property for the Lyapunov exponents? So far, we weren't able to prove neither to present a counter-example to it.

We say that $f$ satisfies the \textit{Anosov Closing property} if there exist $C_1 ,\varepsilon _0 ,\theta >0$ such that if $z\in M$ satisfies $d(f^n(z),z)<\varepsilon _0$ then there exists a periodic point $p\in M$ such that $f^n(p)=p$ and
\begin{displaymath}
d(f^j(z),f^j(p))\leq C_1 e^{-\theta \min\lbrace j, n-j\rbrace}d(f^n(z),z)
\end{displaymath}
for every $j=0,1,\ldots ,n$. Examples of maps satisfying this property are shifts of finite type, basic pieces of Axiom A diffeomorphisms and more generally, hyperbolic homeomorphisms. See for instance, \cite{KH95} p.269, Corollary 6.4.17. 

In what follows we are also going to assume that $A:M\to M(d,\R)$ is an $\alpha$-H\"{o}lder continuous map. This means that there exists a constant $C_2>0$ such that
\begin{displaymath}
\norm{A(x)-A(y)} \leq C_2 d(x,y)^{\alpha}
\end{displaymath}
for all $x,y\in M$ where $\norm{A}$ denotes the operator norm of a matrix $A$, that is, $\norm{A} =\sup \lbrace \norm{Av}/\norm{v};\; \norm{v}\neq 0 \rbrace$. 

\begin{corollary}
Let $f: M\to M $ be a homeomorphism satisfying the Anosov Closing property, $\mu$ an ergodic $f$-invariant probability measure and $A:M\to M(d,\R)$ an $\alpha$-H\"{o}lder continuous map. Then, $(A,f,\mu)$ has the periodic approximation property for the Lyapunov exponents and for the Oseledets splitting.
\end{corollary}
\begin{proof}
It follows from Theorem 2.1 of \cite{Bac} and from its proof that $(A,f,\mu)$ has the periodic approximation property for the Lyapunov exponents. The result then follows applying Theorem \ref{theo: main}.
\end{proof}
As a simple consequence we have 
\begin{corollary} Let $(A,f,\mu)$ be as in the previous corollary. Then, for $\mu$-almost every $x\in M$ there exists a sequence of periodic points $(p_k)_{k\in \mathbb{N}}$ such that
\begin{displaymath}
\measuredangle(E^{j,A}_x,F^{j,A}_{p_k})\xrightarrow{k\to \infty}0
\end{displaymath}
for every $j=1,\ldots ,l$. Moreover, the sequence $(p_k)_{k\in \mathbb{N}}$ may be taken so that
$$\gamma_j (A,p_k)\xrightarrow{k\to \infty}\gamma_j (A,\mu)$$
for every $j=1,\ldots,d$ and
$$\mu _{p_k} =\dfrac{1}{n_k}\sum _{j=0}^{n_k-1}\delta _{f^j(p_k)}\xrightarrow[k\to \infty]{\text{weak}^*}\mu$$
where $n_k$ is the $f$-period of $p_k$.
\end{corollary}

\section{Preliminaries}
In this section we present some preliminary notions and results that are going to be used in the proof of our main theorem. We start by recalling the notion of semi-projective cocycle introduced in \cite{BP}.

\subsection{Semi-projective cocycles} \label{sec: semi-projective} Let $\Pd$ denote the real $(d-1)$-dimensional projective space, that is, the space of all one-dimensional subspaces of $\mathbb{R} ^d$. Given a continuous map $A: M\to M(d,\mathbb{R})$, we want to define an action on $\Pd$ which is, in some sense, induced by $A$. If $(x,[v])\in M\times \Pd$ is such that $A(x)v\neq 0$ then we have a natural action induced by $A$ on $\Pd$ which is just given by $A(x)\left[v\right]=\left[A(x)v\right]$. The difficulty appears when $A(x)v=0$ for some $v\neq 0$. To bypass this issue, let us consider the closed set given by
\begin{displaymath}
\ker(A)=\lbrace (x,[v])\in M\times \Pd ; \;A(x)v=0 \rbrace .
\end{displaymath}
If $\mu(\pi(\ker(A)))=0$ where $\pi :M\times \Pd \to M$ denotes the canonical projection on the first coordinate, then $A(x)$ is invertible for $\mu$-almost every $x\in M$ and hence it naturally induces a map on $\Pd$ which is defined $\mu$-almost everywhere and is all we need. Otherwise, if $\mu(\pi(\ker(A)))>0$ let us consider the set
\begin{displaymath}
K(A)=\lbrace (x,[v])\in M\times \Pd ; \; A^n(x)v=0\mbox{ for some }n>0 \rbrace.
\end{displaymath}
Observe that $K(A)\cap\lbrace x \rbrace \times \Pd \subset \lbrace x \rbrace \times E^{l,A}_x$ for every \textit{regular} point $x\in M$.

Since $\pi(K(A))$ is an $f$-invariant set and $\mu$ is ergodic it follows that $\mu(\pi(K(A)))=1$. Thus, we can define a mensurable section $\sigma:M\to \Pd$ such that $(x,\sigma(x))\in K(A)$. Moreover, we can do this in a way such that if $x\in \pi(\ker(A))$ then $(x,\sigma(x))\in \ker(A)$. Fix such a section. We now define the \textit{semi-projective cocycle} associated to $A$ and $f$ as being the map $F_{A}:M\times \Pd\to M\times \Pd$ given by 
\begin{displaymath}
F_{A} (x,\left[v\right])=\left\lbrace\begin{array}{c} 
                         (f(x),\left[ A(x)v\right]) \mbox{ if }A(x)v\neq 0\\
                         (f(x),\sigma(f(x)) \mbox{ if }A(x)v= 0.\\
                        \end{array}
                        \right.
\end{displaymath}
This is a measurable function which coincides with the usual projective cocycle outside $\ker(A)$. In particular, it is continuous outside $\ker(A)$. From now on, given a non-zero element $v\in \mathbb{R}^{d}$ we are going to use the same notation to denote its equivalence class in $\Pd$. 

Given a measure $m$ on $M\times \Pd$, observe that if $m(\ker(A))=0$ then ${F_A}_{\ast}m$ does not depend on the way the section $\sigma$ was chosen. Indeed, if $\psi:M\times \Pd\to \mathbb{R}$ is a mensurable function then
\begin{displaymath}
 \int_{M\times \Pd}\psi\circ F_A dm= \int_{M\times \Pd \setminus \ker(A)}\psi\circ F_A dm.
\end{displaymath}

In the sequel, we will be primarily interested in $F_A$-invariant measures on $M\times \Pd$ projecting on $\mu$, that is, $\pi _{\ast}m=\mu$ and such that $m(\ker(A))=0$. We start by recalling a result from \cite{BP} which says that if the cocycle $A$ has two different Lyapunov exponents then any such a measure may be written as a convex combination of measures concentrated on a suitable combination of the Oseledets subspaces. In order to state it, let us consider the \emph{Oseledets slow} and \emph{fast} subspaces of `order i' associated to $A$ which are given, respectively, by
$$E_x^{s_i,A}=E^{i+1,A}_x\oplus \cdots \oplus E^{l,A}_x$$
and
$$E_x^{u_i,A}=E^{1,A}_x\oplus \cdots \oplus E^{i,A}_x.$$

\begin{proposition}[Propostion 3.1 of \cite{BP}]\label{prop: decomposition}
If $\gamma_i(A)>\gamma_{i+1}(A)$ then every $F_A$-invariant measure projecting on $\mu$ and such that $m(\ker(A))=0$ is of the form $m=a m^{u_i}+b m^{s_i}$ for some $a,b\in [0,1]$ such that $a+b=1$, where $m^{\ast}$ is an $F_A$-invariant measure projecting on $\mu$ such that its disintegration $\{m^{\ast}_x\}_{x\in M}$ with respect to $\mu$ satisfies $m^{\ast}_x(E^{\ast}_x)=1$ for $\ast\in\lbrace s_i,u_i \rbrace$. 
\end{proposition}

\subsection{The adjoint cocycle}
Given $x\in M$, let $A^{\ast}(x):(\mathbb{R} ^d)^{\ast} \to (\mathbb{R} ^d)^{\ast}$ be the adjoint operator of $A(f^{-1}(x))$ defined by 
\begin{equation}\label{eq: definition adjoint}
(A^{\ast}(x)u)v = u(A(f^{-1}(x))v) \; \mbox{for each} \; u\in (\mathbb{R} ^d)^{\ast} \; \mbox{and} \; v\in \mathbb{R} ^d.
\end{equation}
Fixing some inner product $\langle \; ,\; \rangle$ on $\mathbb{R} ^d$ and identifying the dual space $(\mathbb{R} ^d)^{\ast}$ with $\mathbb{R} ^d$ we get the map $A^{\ast}:M\to M(d,\mathbb{R})$ and equation \eqref{eq: definition adjoint} becomes
\begin{displaymath}
\langle A(f^{-1}(x))u,v\rangle = \langle u,A^{\ast}(x)v\rangle \; \mbox{for every}\; u,v\in \mathbb{R} ^d.
\end{displaymath}
The \textit{adjoint cocycle} of $A$ is then defined as the cocycle generated by the map $A^{\ast}:M\to M(d,\mathbb{R})$ over $f^{-1}:M\to M$.

An useful remark is that the Lyapunov exponents counted with multiplicities of the adjoint cocycle are the same as those of the original cocycle. This follows from the fact that a matrix $B$ and its transpose $B^T$ have the same singular values combined with Kingman's sub-additive theorem. Moreover, Oseledets subspaces of the adjoint cocycle are strongly related with the ones of the original cocycle. More precisely, Lemma 3.1 of \cite{BP} tells us that 
\begin{equation}\label{adjoint.spaces}
E^{s_i,A}_x=(E^{u_i,A^{\ast}}_x)^{\bot} \text{ for every } i=1,\ldots,l
\end{equation}
where the right-hand side denotes the orthogonal complement of the space $E^{u_i,A^{\ast}}_x$.

\subsection{Exterior powers and induced cocycles} 

For every $1\leq j\leq d$ we denote by $\exteriorj$ the $j$th \textit{exterior power} of $\mathbb{R} ^d$ which is the space of alternate $j$-linear forms on the dual $(\mathbb{R} ^d)^{\ast}$. If $\wedge$ denotes the exterior product of vectors of $\mathbb{R} ^d$ then a basis for $\exteriorj$ is given by $\{e_{i_1}\wedge \ldots \wedge e_{i_j}; \; 1\leq i_1<\ldots <i_j\leq l\}$ whenever $\{e_i\}_{i=1}^{d}$ is a basis for $\mathbb{R}^d$. We may also consider the exterior product $V\wedge W$ of subspaces $V$ and $W$ of $\mathbb{R} ^d$. This is defined as the exterior product of the elements of any basis of $V$ with the elements of any basis of $W$. Any linear map $L\in \mbox{M}(d,\mathbb{R})$ induces a linear map $\Lambda ^jL:\exteriorj \to \exteriorj$ by 
\begin{displaymath}
\Lambda ^jL (\omega): \phi _1\wedge \ldots \wedge \phi _j\to \phi _1\circ L\wedge\ldots \wedge\phi _j\circ L. 
\end{displaymath}
Hence, a linear cocyle generated by $B:M\to M(d,\mathbb{R} )$ over $f$ induces a linear cocycle over $f$ on the $j$th exterior power which is generated by the map $x\to \Lambda ^jB(x)$. Moreover, if $B$ satisfies the integrability condition so does $\Lambda ^jB$ and its Lyapunov exponents are given by 
\begin{equation} \label{Lyapunov exterior power}
\{\gamma _{i_1}(B) +\ldots +\gamma _{i_j}(B); \; 1\leq i_1<\ldots < i_j\leq l\}.
\end{equation}
Furthermore, Oseledets subspaces of $\Lambda ^jB$ are strongly related with the ones of $B$. In particular, letting $x\in M$ be a regular point for $(B,f,\mu)$ and $d_i(B)=\sum _{j=1}^i\text{dim}(E^{j,A}_x)$ then for every $1\leq i\leq l$ the Osleledets subspace of $\Lambda ^{d_i(B)}B$ at the point $x\in M$ associated to $\gamma _{1}(B)+\gamma _2(B) +\ldots +\gamma _{d_i(B)}(B)$ is given by
\begin{equation}\label{Oseledets exterior power}
E^{1,B}_x\wedge \ldots \wedge E^{i,B}_x.
\end{equation}
This is all we are going to use about the Oseledets subspaces of induced cocycle.

Let $\grassj$ denote \textit{Grassmannian manifold} of $j$-dimensional subspaces of $\mathbb{R} ^d$. The map $\psi :\grassj \to \Pr (\exteriorj )$ which assigns to each subspace $E\in \grassj$ the projective point $[v]\in \Pr (\exteriorj )$, where $v=v_1\wedge \ldots \wedge v_j$ and $\{v_1,\ldots ,v_j\}$ is any basis for $E$, is an embedding known as the \textit{Pl\"ucker embedding}. Therefore, if $\rho (.,.)$ is a distance on $\Pr (\exteriorj )$ we may push it back to $\grassj$ via $\psi$. More precisely, the map $\d_{\exteriorj}:\grassj\times \grassj\to \mathbb{R}$ given by
\begin{displaymath}
\d_{\exteriorj} (E_1,E_2)=\rho (\psi (E_1),\psi (E_2))
\end{displaymath}
is a distance on $\grassj$ and moreover, if $\rho$ is a distance given by an inner product in the linear space $\exteriorj$ then $\d_{\exteriorj}$ is equivalent to the distance defined in \eqref{def:distancia}.

\section{Approximation of the fastest Oseledets subspace}
In this section we get the desired approximation property for $E^{1,A}_x$ whenever it has dimension one. We chose to present this case separately because its proof illustrates the main ideas used in the general case and, moreover, notations are simpler providing a cleaner exposition.

\begin{proposition} \label{prop: fastest}
Assume $(A,f,\mu)$ has the periodic approximation property for the Lyapunov exponents and let $(p_k)_{k\in \mathbb{N}}$ be a sequence of periodic points satisfying \eqref{eq: mu_pk to mu} and \eqref{eq: gamma pk to gamma mu}. Assume also that $\text{dim}(E^{1,A}_x)=1$. Then given $\varepsilon >0$ there exist an arbitrarily large $k\in \mathbb{N}$ and a set $\G^1:=\G^1_\varepsilon \subset M$ with $\mu (\G^1)>1-\varepsilon$ so that for every $x\in \G^1$ there exists $q\in \text{orb}(p_k)$ satisfying
$$\measuredangle(E^{1,A}_x,E^{1,A}_{q})<\varepsilon.$$
\end{proposition}
\begin{proof}

We start observing that as
$$\gamma_j (A,p_k)\xrightarrow{k\to \infty}\gamma_j (A,\mu)$$
for every $j=1,\ldots,d$ and $\gamma_1(A,\mu)>\gamma_2(A,\mu)$ it follows that $\gamma _1(A,p_k)>\gamma _2(A,p_k)$ for every $k$ sufficiently large and thus $E^{1,A}_{p_k}$ is also one-dimensional. Let us assume without loss of generality that this is indeed the case for every $k\in \mathbb{N}$.

For each $k\in \mathbb{N}$, let us consider the measure
\begin{displaymath}
m_k=\int _{M} \delta _{(x,E^{1,A}_{x})} d\mu_{p_k}(x)
\end{displaymath}
and let $m$ be the measure given by
\begin{displaymath}
m=\int _{M} \delta _{(x,E^{1,A}_{x})} d\mu(x).
\end{displaymath}
Observe that these are $F_A$-invariant measures on $M\times \Pd$ concentrated on $\{(x,E^{1,A}_x); x\in M\}$ and projecting to $\mu _{p_k}$ and $\mu$, respectively. Consequently, letting $\varphi _{A}: M\times \Pd \to \mathbb{R}$ be the map given by
\begin{displaymath}
\varphi _{A}(x,v)=\log \frac{\parallel A(x)v\parallel }{\parallel v \parallel},
\end{displaymath}
it follows easily from the definition and Birkhoff's ergodic theorem that
\begin{equation}\label{eq: measure realizing Lyap exp Ak}
\gamma _1(A,\mu _{p_k})=\int _{M\times \Pd} \varphi _{A}(x,v)dm_k
\end{equation}
and
\begin{equation}\label{eq: measure realizing Lyap exp A}
\gamma _1(A,\mu)=\int _{M\times \Pd} \varphi _{A}(x,v)dm.
\end{equation}

We claim now that $m_k$ converges to $m$ in the weak$^*$ topology. Indeed, let $\{m_{k_i}\}_{i\in \mathbb{N}}$ be a convergent subsequence of $\{m_{k}\}_{k\in \mathbb{N}}$ and suppose it converges to $\tilde{m}$. Since $M \times \Pd$ is a compact space it suffices to prove that $\tilde{m}=m$. In order to do so, we need the following auxiliary result.

\begin{lemma}\label{lemma: kernel has zero measure}
The measure $\tilde{m}$ satisfies $\tilde{m}(\ker (A))=0$. Moreover, it is $F_A$-invariant.
\end{lemma}

\begin{proof}
Suppose by contradiction that $\tilde{m}(\ker(A))=2 c>0$. For each $\delta >0$ let us consider 
\begin{displaymath}
K_{\delta}=\left\{ (x,v)\in M\times \Pd ; \; \norm{A(x)\frac{v}{\norm{v}}}<\delta \right\}.
\end{displaymath}
These are open sets such that $\ker(A)=\cap_{\delta >0} K_{\delta}$ and $\tilde{m}(K_{\delta})\geq \tilde{m}(\ker(A))>c>0$. 

Fix $b\in \mathbb{R}$ such that
$$b< \gamma_1(A,\mu)-\sup_{x,\norm{v}=1} \log \norm{A(x)v}$$ 
and let $\delta>0$ be such that $\log y<\frac{b}{c}$ for every $y<\delta$. Then, for every $i$ sufficiently large $m_{k_i}(K_{\delta})>c>0$ and consequently

\begin{displaymath}
\begin{split}
\gamma_1(A,p_{k_i})&=\int \varphi_{A}d m_{k_i} =\int_{K_\delta} \varphi_{A}d m_{k_i}+ \int _{K_\delta ^c} \varphi_{A}d m_{k_i} \\
&<b+\sup_{x,\norm{v}=1} \log \norm{A(x)v}
\end{split}
\end{displaymath}
contradicting the choice of $b$. Thus, $\tilde{m}(\ker(A))=0$ as we want.

To prove that $\tilde{m}$ is $F_A$-invariant one only has to show that, given a continuous map $\psi:M\times \Pd\to \mathbb{R}$, 
\begin{equation} \label{eq: FA invariance}
\lim_{i\to \infty} \int \psi\circ F_{A} dm_{k_i}=\int \psi\circ F_{A} d\tilde{m}.
\end{equation}

Given $\varepsilon >0$ let $\delta >0$ be small enough so that $\tilde{m}(\overline{K_\delta})<\frac{\varepsilon}{\norm{\psi}}$. Let $\hat{\psi}:M\times \Pd\to \mathbb{R}$ be a continuous function such that it coincides with $\psi\circ F_{A}$ outside $K_{\delta}$ and $\norm{\hat{\psi}}\leq \norm{\psi}$. Note that the existence of such a map is guaranteed by Tietze's extension theorem. Then, 
\begin{displaymath}
\begin{split}
\left| \int \psi\circ F_{A} dm_{k_i}-\int \psi\circ F_{A} d\tilde{m} \right| &\leq \left| \int _{K_\delta ^c} \hat{\psi}dm_{k_i}-\int_{K_\delta ^c} \hat{\psi}d\tilde{m}\right| \\
&+ \int _{K_\delta }\mid \psi \circ F_A\mid dm_{k_i}+ \int_{K_\delta } \mid \psi \circ F_A\mid d\tilde{m}\\
&< 4\varepsilon
\end{split}
\end{displaymath}
for every $i$ sufficiently large proving \eqref{eq: FA invariance} and consequently the lemma.
\end{proof}

Now, recalling that 
\begin{displaymath}
\gamma _1(A,p_{k_i})\xrightarrow{i\to +\infty}\gamma _1(A,\mu)
\end{displaymath}
and observing that
\begin{displaymath}
\int _{M\times \Pd} \varphi _{A}(x,v)dm_{k_i}\xrightarrow{i\to +\infty} \int _{M\times \Pd} \varphi _{A}(x,v)d\tilde{m}
\end{displaymath}
it follows from \eqref{eq: measure realizing Lyap exp Ak} that 
\begin{equation}\label{eq: uniq measure realizing}
\gamma _1(A,\mu)=\int _{M\times \Pd} \varphi _{A}(x,v)d\tilde{m}.
\end{equation}
Thus, from Proposition \ref{prop: decomposition} we get that $\tilde{m}=m$ as claimed. In fact, otherwise the referred proposition would give us that $\tilde{m}=am^{1}+bm^{s}$ where $a,b\in (0,1)$ are such that $a+b=1$ and $m^{s}$ is an $F_A$-invariant measure concentrated on $\{(x,E^{2}_x\oplus \cdots \oplus E^l_x); \; x\in M\}$.  
Therefore,
\begin{displaymath}
\begin{split}
\gamma _1(A,\mu)&=\int _{M\times \Pd} \varphi _{A}(x,v)d\tilde{m} \\
&=a\int _{M\times \Pd} \varphi _{A}(x,v)dm^{1} +b\int _{M\times \Pd} \varphi _{A}(x,v)dm^{s} \\
&\leq a \gamma _1(A,\mu)+b\gamma _2(A,\mu) < \gamma _1(A,\mu).
\end{split}
\end{displaymath}

Observe that as a consequence of this argument we also get that $m$ is the only $F_A$-invariant measure projecting on $\mu$ and satisfying \eqref{eq: uniq measure realizing}. This is going to be used on Section \ref{sec: simultaneous}.

Given $\varepsilon >0$, let $\G^1\subset \Reg \cap \text{supp}(\mu)$ be a compact set satisfying $\mu(\G^1)>1-\varepsilon$ and such that
$$x\to E^{1,A}_x\oplus \ldots \oplus  E^{l,A}_x$$
is continuous on $x\in \G^1$. Observe that the existence of such set is guaranteed by Lusin's theorem. 

Fix $x\in \G^1$. Let $B((x,E^{1,A}_x),\varepsilon)$ be the open $\varepsilon$-neighborhood of $(x,E^{1,A}_x)$ on $M\times \Pd$. Recall that we are considering $M\times \Pd$ endowed with the metric $\tilde{d}$ given by $\tilde{d}((y,v),(z,w))=d(y,z)+\measuredangle(v,w)$.

Thus, since $m(B((x,E^{1,A}_x),\varepsilon))>0 $ (recall that $y\to E^{1,A}_y$ is continuous when restricted to an arbitrarily $\mu$-large set and $\mu$ gives positive measure to every open ball centered at $x$) and $m_k\to m$ it follows that
$$\liminf_{k\to \infty} m_k(B((x,E^{1,A}_x),\varepsilon))\geq m(B((x,E^{1,A}_x),\varepsilon))>0.$$
In particular, there exists $k_0(x)\in \mathbb{N}$ such that
$m_k(B((x,E^{1,A}_x),\varepsilon))>0$ for every $k\geq k_0(x)$. Consequently, it follows from the definition of $m_{k}$ that for every $k\geq k_0(x)$ there exists $j_k\in\{0,1,\ldots,n_k-1\}$ so that 
$$\measuredangle(E^{1,A}_x, E^{1,A}_{f^{j_k}(p_k)})<\varepsilon.$$
To conclude the proof it remains to observe that $k_0(x)$ may be taken independent of $x\in \G^1$. But this follows easily using that $\G^1$ is compact and the Oseledets splitting is continuous when restricted to it. 

\end{proof}

\section{Simultaneous approximations}\label{sec: simultaneous}
In this section we get the desired approximations for the Oseledets slow and fast subspaces. Recall the definitions of $E^{u_j,A}_x$ and $E^{s_j,A}_x$ given in Section \ref{sec: semi-projective}.
\begin{proposition}\label{prop: simultaneous}
Assume $(A,f,\mu)$ has the periodic approximation property for the Lyapunov exponents and let $(p_k)_{k\in \mathbb{N}}$ be a sequence of periodic points satisfying \eqref{eq: mu_pk to mu} and \eqref{eq: gamma pk to gamma mu}. Then given $\varepsilon >0$ there exist an arbitrarily large $k\in \mathbb{N}$ and a set $\G^{s,u}:=\G^{s,u}_\varepsilon \subset M$ with $\mu (\G^{s,u})>1-\varepsilon$ so that for every $x\in \G^{s,u}$ there exists $q\in \text{orb}(p_k)$ satisfying
$$\measuredangle(E^{u_j,A}_x,F^{u_j,A}_{q})<\varepsilon$$
and 
$$\measuredangle(E^{s_j,A}_x,F^{s_j,A}_{q})<\varepsilon$$
for every $j\in\{1,\ldots,l \}$.
\end{proposition}

\begin{proof} 
We start observing that if $j=l$ then $E^{u_l,A}_x=\mathbb{R}^d=F^{u_l,A}_q$. So, we only have to prove the proposition for $j<l$.

For each $j=1,\ldots,l$ we set $d_j=\sum _{i=1}^j\text{dim}(E^{i,A}_x)$ and similarly $d^\ast_j=\sum _{i=1}^j\text{dim}(E^{i,A^\ast}_x)$. Let $\Lambda$ be the space 
$$\Lambda ^{d_1}(\mathbb{R}^d)\times \Lambda ^{d_2}(\mathbb{R}^d)\times \ldots\times \Lambda ^{d_{l-1}}(\mathbb{R}^d)\times \Lambda ^{d^*_1}(\mathbb{R}^d)\times \Lambda ^{d^*_2}(\mathbb{R}^d)\times \ldots\times \Lambda ^{d^*_{l-1}}(\mathbb{R}^d) $$
and let $\Pr$ be equal to
$$\Pr(\Lambda ^{d_1}(\mathbb{R}^d))\times \Pr(\Lambda ^{d_2}(\mathbb{R}^d))\times \ldots\times \Pr(\Lambda ^{d_{l-1}}(\mathbb{R}^d))\times \Pr(\Lambda ^{d^*_1}(\mathbb{R}^d))\times \Pr(\Lambda ^{d^*_2}(\mathbb{R}^d))\times \ldots\times \Pr(\Lambda ^{d^*_{l-1}}(\mathbb{R}^d)). $$

The map $f_{\Lambda A}:M\times\Lambda \to M\times \Lambda$ which assigns to each point 
$$(x,v_{d_1},\ldots,v_{d_{l-1}},v_{d^*_{_1}},\ldots,v_{d^*_{l-1}})\in M\times \Lambda$$
the point 
$$(f(x),\Lambda ^{d_{1}}A(x)v_{d_1},\ldots,\Lambda ^{d_{l-1}}A(x)v_{d_{l-1}},\Lambda ^{d^*_{1}}A^*(x)v_{d^*_{_1}},\ldots,\Lambda ^{d^*_{l-1}}A^*(x)v_{d^*_{l-1}})\in M\times \Lambda$$
induces a semi-projective cocycle $F_{\Lambda A}:M\times \Pr\to M\times \Pr$ as described in Section \ref{sec: semi-projective}.

For each $k\in \mathbb{N}$, let us consider the measure 
\begin{displaymath}
m_k=\int _{M} \delta _{(x,E^{1,\Lambda^{d_1}A}_{x},E^{1,\Lambda^{d_2}A}_{x},\ldots ,E^{1,\Lambda^{d_{l-1}}A}_{x},E^{1,\Lambda^{d^*_1}A^*}_{x},E^{1,\Lambda^{d^*_2}A^*}_{x},\ldots ,E^{1,\Lambda^{d^*_{l-1}}A^*}_{x})} d\mu_{p_k}(x)
\end{displaymath}
where $\mu_{p_k}$ is as in \eqref{eq: mu_pk to mu} and let $m$ be the measure given by
\begin{displaymath}
m=\int _{M} \delta _{(x,E^{1,\Lambda^{d_1}A}_{x},E^{1,\Lambda^{d_2}A}_{x},\ldots ,E^{1,\Lambda^{d_{l-1}}A}_{x},E^{1,\Lambda^{d^*_1}A^*}_{x},E^{1,\Lambda^{d^*_2}A^*}_{x},\ldots ,E^{1,\Lambda^{d^*_{l-1}}A^*}_{x})}  d\mu(x).
\end{displaymath}
Observe that from the choice of $d_j$ and $d^*_j$ and \eqref{Lyapunov exterior power} and \eqref{Oseledets exterior power} these are well defined measures on $M\times \Pr$ and, moreover, they are $F_{\Lambda A}$-invariant measures concentrated on 
$$\{(x,E^{1,\Lambda^{d_1}A}_{x},E^{1,\Lambda^{d_2}A}_{x},\ldots ,E^{1,\Lambda^{d_{l-1}}A}_{x},E^{1,\Lambda^{d^*_1}A^*}_{x},E^{1,\Lambda^{d^*_2}A^*}_{x},\ldots ,E^{1,\Lambda^{d^*_{l-1}}A^*}_{x}); x\in M\}$$ 
and projecting to $\mu _{p_k}$ and $\mu$, respectively. Consequently, letting $\varphi _{\Lambda A}: M\times \Pr \to \mathbb{R}$ be the map given by
\begin{displaymath}
\begin{split}
\varphi _{\Lambda A}(x,v_{d_1},\ldots,v_{d_{l-1}},v_{d^*_{_1}},\ldots,v_{d^*_{l-1}})&=\log \frac{\parallel \Lambda ^{d_1}A(x)v_{d_1}\parallel }{\parallel v_{d_1} \parallel}+\ldots+\log \frac{\parallel \Lambda ^{d_{l-1}}A(x)v_{d_{l-1}}\parallel }{\parallel v_{d_{l-1}} \parallel}\\
&+\log \frac{\parallel \Lambda ^{d^*_1}A^*(x)v_{d^*_1}\parallel }{\parallel v_{d^*_1} \parallel}+\ldots+\log \frac{\parallel \Lambda ^{d^*_{l-1}}A^*(x)v_{d^*_{l-1}}\parallel }{\parallel v_{d^*_{l-1} }\parallel},
\end{split}
\end{displaymath}
it follows from the definition and Birkhoff's ergodic theorem that
\begin{equation*}
\int _{M\times \Pr} \varphi _{\Lambda A}(x,v_{d_1},\ldots,v_{d_{l-1}},v_{d^*_{_1}},\ldots,v_{d^*_{l-1}})dm_k = 2\sum_{j=1}^{l-1} \gamma _1(\Lambda ^{d_j}A,\mu_{p_k})
\end{equation*}
and
\begin{equation*}
\int _{M\times \Pr} \varphi _{\Lambda A}(x,v_{d_1},\ldots,v_{d_{l-1}},v_{d^*_{_1}},\ldots,v_{d^*_{l-1}})dm= 2\sum _{j=1}^{l-1}\gamma _1(\Lambda^{d_j} A,\mu).
\end{equation*}

We now observe that $m_k$ converges to $m$ in the weak$^{*}$ topology. Indeed, suppose $m_k$ converges to some measure $\tilde{m}$. Then, using the previous observations and proceeding as we did in the proof of Proposition \ref{prop: fastest} we get that $\tilde{m}$ is an $F_{\Lambda A}$-invariant measure on $M\times \Pr$ projecting to $\mu$ and satisfying 
$$
\int _{M\times \Pr} \varphi _{\Lambda A}(x,v_{d_1},\ldots,v_{d_{l-1}},v_{d^*_{_1}},\ldots,v_{d^*_{l-1}})d\tilde{m}= 2\sum _{j=1}^{l-1}\gamma _1(\Lambda^{d_j} A,\mu).
$$
Now, the claim follows easily from our next result.

\begin{lemma} \label{lem: uniqueness}
If $\tilde{m}$ is an $F_{\Lambda A}$-invariant measure on $M\times \Pr$ projecting to $\mu$ such that 
\begin{equation*}
\int _{M\times \Pr} \varphi _{\Lambda A}(x,v_{d_1},\ldots,v_{d_{l-1}},v_{d^*_{_1}},\ldots,v_{d^*_{l-1}})d\tilde{m}= 2\sum _{j=1}^{l-1}\gamma _1(\Lambda^{d_j} A,\mu)
\end{equation*}
then $\tilde{m}=m$.
\end{lemma}
In order to prove this lemma we are going to use the following simple fact 
\begin{claim}\label{claim: product measure}
Let $M_1\times M_2$ be a product space and for $j=1,2$ let $\pi_j:M_1\times M_2\to M_j$ be the canonical projection on $M_j$. If $\xi$ is a measure on $M_1\times M_2$ and $(\pi_1)_*\xi =\delta _{x_1}$ and $(\pi_2)_*\xi =\delta _{x_2}$ for some $x_1\in M_1$ and $x_2\in M_2$ then $\xi=\delta_{x_1}\times \delta_{x_1}$. Moreover, a similar statement holds for measures on a product of any finite number of spaces.
\end{claim} 
Indeed, observing that 
$$1=(\pi_1)_*\xi(\{x_1\})=\xi(\{x_1\}\times M_2)$$
we get that $\text{supp } \xi\subset \{x_1\}\times M_2$. Similarly, we conclude that $\text{supp } \xi\subset M_1\times \{x_2\}$. Thus, $\text{supp } \xi\subset \{x_1\}\times \{x_2\}$ and $\xi =\delta _{x_1}\times \delta _{x_2}$.

\begin{proof}[Proof of Lemma \ref{lem: uniqueness}]
Let $\tilde{m}=\int _M \tilde{m}_x d\mu(x)$ be a disintegration of $\tilde{m}$ along $\{\{x\}\times \Pr\}_{x\in M}$ and for each $j\in \{1,\ldots ,l-1\}$ let $\pi _j:\Pr\to \Pr(\Lambda ^{d_j}(\mathbb{R}^d))$ and $\pi ^*_j:\Pr\to \Pr(\Lambda ^{d^*_j}(\mathbb{R}^d))$ be the canonical projections on $\Pr(\Lambda ^{d_j}(\mathbb{R}^d))$ and $\Pr(\Lambda ^{d^*_j}(\mathbb{R}^d))$, respectively, and $\nu ^j_x=(\pi _j)_{*}\tilde{m}_x$ and $\nu ^{j*}_x=(\pi ^*_j)_{*}\tilde{m}_x$ be the projections of $\tilde{m}_x$ on $ \Pr(\Lambda ^{d_j}(\mathbb{R}^d))$ and $\Pr(\Lambda ^{d^*_j}(\mathbb{R}^d))$, respectively. We claim now that, for each $j\in \{1,\ldots ,l-1\}$, $\nu^j_x=\delta _{E^{1,\Lambda^{d_j}A}_x}$ and $\nu^{j*}_x=\delta _{E^{1,\Lambda^{d^*_j}A^*}_x}$ for $\mu$-almost every $x\in M$. Given $j\in \{1,\ldots ,l-1\}$ let $F_{\Lambda ^{d_j}A}$ denote the semi-projective cocycle induced by $(f, \Lambda ^{d_j}A)$ on $M\times \Pr(\Lambda ^{d_j}(\mathbb{R}^d))$. Similarly we define $F_{\Lambda ^{d^*_j}A^*}$.

Let us consider the measures
$$\nu_j=\int _M\nu^j_xd\mu(x) \text{ and }\nu_j^{*}=\int _M\nu^{j*}_xd\mu(x).$$
These are $F_{\Lambda ^{d_j}A}$ and $F_{\Lambda ^{d^*_j}A^*}$-invariant measures on $M\times \Pr(\Lambda ^{d_j}(\mathbb{R}^d))$ and $M\times \Pr(\Lambda ^{d^*_j}(\mathbb{R}^d))$, respectively, projecting to $\mu$ and satisfying 
$$\gamma _1(\Lambda ^{d_j}A,\mu)=\int \log \frac{\parallel\Lambda ^{d_j}A(x)v_{d_j}\parallel}{\parallel v_{d_j}\parallel}d\nu ^j \text{ and }\gamma _1(\Lambda ^{d^*_j}A^*,\mu)=\int \log \frac{\parallel\Lambda ^{d^*_j}A^*(x)v_{d^*_j}\parallel}{\parallel v_{d^*_j}\parallel}d\nu ^{j*}.$$
Indeed,
\begin{displaymath}
\begin{split}
2\sum _{j=1}^{l-1}\gamma _1(\Lambda^{d_j} A,\mu)&= \int _{M\times \Pr} \varphi _{\Lambda A}(x,v_{d_1},\ldots,v_{d_{l-1}},v_{d^*_{_1}},\ldots,v_{d^*_{l-1}})d\tilde{m}\\
&=\int _{M}\int_{\Pr} \varphi _{\Lambda A}(x,v_{d_1},\ldots,v_{d_{l-1}},v_{d^*_{_1}},\ldots,v_{d^*_{l-1}})d\tilde{m}_xd\mu (x)\\
&=\sum_{j=1}^{l-1} \int _{M}\int_{\Pr} \log \frac{\parallel \Lambda ^{d_j}A(x)v_{d_j}\parallel}{\parallel v_{d_j}\parallel} d\tilde{m}_xd\mu (x)\\
&+\sum_{j=1}^{l-1} \int _{M}\int_{\Pr} \log \frac{\parallel \Lambda ^{d^*_j}A^*(x)v_{d^*_j}\parallel}{\parallel v_{d^*_j}\parallel} d\tilde{m}_xd\mu (x)\\
&\leq \sum_{j=1}^{l-1} \int _{M}\int_{\Pr(\Lambda^{d_j}(\mathbb{R}^d))} \log \frac{\parallel \Lambda ^{d_j}A(x)v_{d_j}\parallel}{\parallel v_{d_j}\parallel} d\nu^j_xd\mu (x)\\
&+\sum_{j=1}^{l-1} \int _{M}\int_{\Pr(\Lambda^{d^*_j}(\mathbb{R}^d))} \log \frac{\parallel \Lambda ^{d^*_j}A^*(x)v_{d^*_j}\parallel}{\parallel v_{d^*_j}\parallel} d\nu^{j*}_xd\mu (x).\\
\end{split}
\end{displaymath}
Thus, since 
$$ \int _{M \times \Pr(\Lambda^{d_j}(\mathbb{R}^d))} \log \frac{\parallel\Lambda ^{d_j}A(x)v_j\parallel}{\parallel v_j\parallel} d\xi^j\leq \gamma _1(\Lambda ^{d_j}A,\mu)$$
for every $F_{\Lambda ^{d_j}A}$-invariant measure $\xi^j$ projecting on $\mu$ and similarly for every $F_{\Lambda ^{d^*_j}A^*}$-invariant measure projecting on $\mu$ our claim follows. Hence, from the uniqueness obtained in the proof of Proposition \ref{prop: fastest} we get that 
$$\nu_j=\int_M \delta _{(x,E^{1,\Lambda^{d_j}A}_x)}d\mu(x) \text{ and } \nu^*_j=\int_M \delta _{(x,E^{1,\Lambda^{d^*_j}A^*}_x)}d\mu(x).$$
Consequently, since a disintegration is essentially unique we get that $\nu^j_x=\delta _{E^{1,\Lambda^{d_j}A}_x}$ and $\nu^{j*}_x=\delta _{E^{1,\Lambda^{d^*_j}A^*}_x}$ for $\mu$-almost every $x\in M$ as claimed. Now, invoking Claim \ref{claim: product measure} we get that 
\begin{equation*}
\begin{split}
\tilde{m}_x &=\delta _{E^{1,\Lambda^{d_1}A}_x}\times \delta _{E^{1,\Lambda^{d_2}A}_x}\times \ldots \times \delta _{E^{1,\Lambda^{d_{l-1}j}A}_x}\times \delta _{E^{1,\Lambda^{d^*_1}A^*}_x}\times\ldots \delta _{E^{1,\Lambda^{d^*_{l-1}}A^*}_x} \\
&=\delta _{(x,E^{1,\Lambda^{d_1}A}_{x},E^{1,\Lambda^{d_2}A}_{x},\ldots ,E^{1,\Lambda^{d_l-1}A}_{x},E^{1,\Lambda^{d^*_1}A^*}_{x},E^{1,\Lambda^{d^*_2}A^*}_{x},\ldots ,E^{1,\Lambda^{d^*_{l-1}}A^*}_{x})}
\end{split}
\end{equation*}
for $\mu$-almost every $x\in M$ and thus $\tilde{m}=m$ as stated.
\end{proof}

Now, using that $m_k$ converges to $m$ and proceeding as we did at the end of the proof of Proposition \ref{prop: fastest} we conclude that, given $\varepsilon '>0$, there exist an arbitrarily large $k\in \mathbb{N}$ and a set $\G^{s,u}:=\G^{s,u}_{\varepsilon '} \subset M$ with $\mu (\G^{s,u})>1-\varepsilon '$ so that for every $x\in \G^{s,u}$ there exists $q\in \text{orb}(p_k)$ satisfying
$$\measuredangle(E^{1,\Lambda ^{d_j}A}_x,F^{1,\Lambda ^{d_j}A}_{q})<\varepsilon'$$
and 
$$\measuredangle(E^{1,\Lambda^{d_j^*} A_*}_x,F^{1,\Lambda^{d^*_j} A_*}_{q})<\varepsilon'$$
for every $j\in\{1,\ldots,l-1 \}$. Thus, recalling \eqref{Oseledets exterior power} we get that
$$\measuredangle(E^{1,A}_x\wedge \ldots \wedge E^{j,A}_x,F^{1,A}_q\wedge \ldots \wedge F^{j,A}_q)<\varepsilon '$$
and 
$$\measuredangle(E^{1,A^*}_x\wedge \ldots \wedge E^{j,A^*}_x,F^{1,A^*}_q\wedge \ldots \wedge F^{j,A^*}_q)<\varepsilon '$$
for every $j\in\{1,\ldots,l-1 \}$. Consequently, from the definition of $\d_{\Lambda ^r(\mathbb{R} ^d)}$ it follows that
$$\d_{\Lambda ^r(\mathbb{R} ^d)}(E^{1,A}_x\oplus \ldots \oplus E^{j,A}_x,F^{1,A}_q\oplus \ldots \oplus F^{j,A}_q)<\varepsilon '$$
and 
$$\d_{\Lambda ^r(\mathbb{R} ^d)}(E^{1,A^*}_x\oplus \ldots \oplus E^{j,A^*}_x,F^{1,A^*}_q\oplus \ldots \oplus F^{j,A^*}_q)<\varepsilon '$$
for every $j\in\{1,\ldots,l-1 \}$.
Now, using the fact that the distances $\d_{\Lambda ^r(\mathbb{R} ^d)}$ and $\d$ are equivalent and taking $\varepsilon '>0$ smaller if necessary, it follows that 
$$\measuredangle(E^{1,A}_x\oplus \ldots \oplus E^{j,A}_x,F^{1,A}_q\oplus \ldots \oplus F^{j,A}_q)<\varepsilon $$
and 
$$\measuredangle(E^{1,A^*}_x\oplus \ldots \oplus E^{j,A^*}_x,F^{1,A^*}_q\oplus \ldots \oplus F^{j,A^*}_q)<\varepsilon $$
for every $j\in\{1,\ldots,l-1 \}$. Summarizing, given $\varepsilon >0$ there exist an arbitrarily large $k\in \mathbb{N}$ and a set $\G^{s,u}:=\G^{s,u}_{\varepsilon } \subset M$ with $\mu (\G^{s,u})>1-\varepsilon $ so that for every $x\in \G^{s,u}$ there exists $q\in \text{orb}(p_k)$ satisfying
$$\measuredangle(E^{u_j,A}_x,F^{u_j,A}_{q})<\varepsilon$$
and 
$$\measuredangle(E^{u_j, A^*}_x,F^{u_j, A^*}_{q})<\varepsilon$$
for every $j\in\{1,\ldots,l-1 \}$. Thus, to conclude the proof of the proposition it remains to recall \eqref{adjoint.spaces} which says that $E^{s_j,A}_x=(E^{u_j,A^*})^{\perp}$.
\end{proof}

\section{Conclusion of the proof}

Let $(p_k)_{k\in \mathbb{N}}$ be a sequence of periodic points satisfying \eqref{eq: mu_pk to mu} and \eqref{eq: gamma pk to gamma mu}. It is easy to see that in order to complete the proof of Theorem \ref{theo: main} it is enough to observe that given $\varepsilon >0$ there exist an arbitrarily large $k\in \mathbb{N}$ and a set $\G:=\G_\varepsilon \subset M$ with $\mu (\G)>1-\varepsilon$ so that for every $x\in \G$ there exists $q\in \text{orb}(p_k)$ satisfying
$$\measuredangle(E^{j,A}_x,F^{j,A}_{q})<\varepsilon$$
for every $j=1,\ldots ,l$. So, this is what we are going to do.

The \textit{cone} of radius $\alpha >0$ around a subspace $V$ of $\mathbb{R} ^d$ is defined as  
\begin{displaymath}
C_{\alpha}(V)=\left\{ w_1+w_2 \in V \ \oplus V^{\perp}; \; \norm{w_2}<\alpha \norm{w_1}\right\}.
\end{displaymath}
Observe that this is equivalent to 
\begin{displaymath}
C_{\alpha}(V)=\left\{ w\in \mathbb{R} ^d ; \; \d\left(\frac{w}{\norm{w}},V\right)<\alpha\right\}
\end{displaymath}
where $\d$ is the distance defined in \eqref{def:distancia}.

\begin{lemma}[Lemma 4.2 of \cite{BP}]\label{intersection}
Given $1\leq j\leq l$, $\varepsilon '>0$ and $\delta>0$ there exist a subset $K=K(\varepsilon ')\subset M$ with $\mu(K)>1-\varepsilon '$ and $\delta'=\delta'(\varepsilon ',\delta)\in (0,\delta )$, such that for every $x\in K$, 
$$ C_{\delta'}(E^{u_j,A}_x)\cap C_{\delta'}(E^{s_{j-1},A}_x)\subset C_{\delta}(E^{j,A}_x).$$
\end{lemma}

Given $\varepsilon >0$, take $\varepsilon' , \delta \in (0,\frac{\varepsilon}{10})$ and let $K(\varepsilon ')\subset M$ and $\delta '>0$ be given by the previous lemma. Define $\mathcal{G}:=\mathcal{G}^{s,u}_{\delta '}\cap K(\varepsilon ')$ where $\mathcal{G}^{s,u}_{\delta '}$ is the set associated to $\delta'$ by Proposition \ref{prop: simultaneous}. Then, $\mu (\mathcal{G})>1-\varepsilon$ and for every $x\in \mathcal{G}$ there exists $q\in \text{orb}(p_k)$ satisfying
$$\measuredangle(E^{u_j,A}_x,F^{u_j,A}_{q})<\delta '$$
and 
$$\measuredangle(E^{s_j, A}_x,F^{s_j, A}_{q})<\delta '$$
for every $j\in\{1,\ldots,l \}$ and moreover
$$ C_{\delta'}(E^{u_j,A}_x)\cap C_{\delta'}(E^{s_{j-1},A}_x)\subset C_{\delta}(E^{j,A}_x).$$
Thus,
$$F^{j,A}_q\subset F^{u_j,A}_q\cap F^{s_{j-1},A}_q\subset C_{\delta'}(E^{u_j,A}_x)\cap C_{\delta'}(E^{s_{j-1},A}_x)\subset C_{\delta}(E^{j,A}_x).$$
Consequently,
$$\measuredangle(E^{j, A}_x,F^{j, A}_{q})<\delta <\varepsilon$$
for every $j\in\{1,\ldots,l \}$ as we wanted. The proof of Theorem \ref{theo: main} is now complete.

\medskip{\bf Acknowledgements.} The author was partially supported by a CAPES-Brazil postdoctoral fellowship under Grant No. 88881.120218/2016-01 at the University of Chicago. 


\end{document}